\documentclass[leqno,12pt]{amsart}
\usepackage{amsmath,amstext,amssymb,amsopn,amsthm,mathrsfs}
\usepackage{epsfig,graphics,latexsym,graphicx,color}
\definecolor{darkgreen}{rgb}{0,.5,0}
\numberwithin{equation}{section}

\allowdisplaybreaks

\newtheorem{theorem}{Theorem}[section]

\newtheorem{propo}[theorem]{Proposition}
\newtheorem{corollary}{Corollary}[section]

\newtheorem*{rem*}{Remark}
\newtheorem{remark}{Remark}[section]

\begin{document}
%\begin{CJK}{Bg5}{ttkai}
%%%%%%%%%%%%%%%%%%%%%%%%%%%%%%%%%%%%%%%%%%%%%%

\title[]{Another characterizations of Muckenhoupt $A_{p}$ class}

%%%%%%%%%%%%%%%%%%%%%%%%%%%%%%%%%%%%%%%%%%%%%%
\author[]{Dinghuai Wang and Jiang Zhou$^\ast$}
\address{College of Mathematics and System Sciences \endgraf
         Xinjiang University \endgraf
         Urumqi 830046 \endgraf
         Republic of China}
\email{Wangdh1990@126.com; zhoujiangshuxue@126.com}
\thanks{The research was supported by National Natural Science Foundation
of China (Grant No.11261055). \\ \qquad * Corresponding author, Email: zhoujiangshuxue@126.com. Telephone:+8613009688600}

%%%%%%%%%%%%%%%%%%%%%%%%%%%%%%%%%%%%%%%%%%%%%%
\begin{abstract}
This manuscript addresses Muckenhoupt $A_{p}$ weight theory in connection to Morrey and BMO
spaces. It is proved that $\omega$ belongs to Muckenhoupt $A_{p}$ class, if and only if Hardy-Littlewood maximal function $M$ is bounded from weighted Lebesgue spaces $L^{p}(\omega)$ to weighted Morrey spaces $M^{p}_{q}(\omega)$ for $1<q< p<\infty$. As a corollary, if $M$ is (weak) bounded on $M^{p}_{q}(\omega)$, then $\omega\in A_{p}$. The $A_{p}$ condition also characterizes the boundedness of the Riesz transform $R_{j}$ and convolution operators $T_{\epsilon}$ on weighted Morrey spaces. Finally, we show that $\omega\in A_{p}$ if and only if $\omega\in \mathrm{BMO}^{p'}(\omega)$ for $1\leq p< \infty$ and $1/p+1/p'=1$.
\end{abstract}
\maketitle

%%%%%%%%%%%%%%%%%%%%%%%%%%%%%%%%%%%%%%%%%%%%%%
%%%%%%%%%%%%%%%%%%%%%%%%%%%%%%%%%%%%%%%%%%%%%%
\maketitle

%%%%%%%%%%%%%%%%%%%%%%%%%%%%%%%%%%%%%%%%%%%%%%
%%%%%%%%%%%%%%%%%%%%%%%%%%%%%%%%%%%%%%%%%%%%%%

\vspace{0.3cm}

\section{Introduction}

For $1< p<\infty$ and a nonnegative locally integrable function $\omega$ on $\mathbb{R}^n$, $\omega$ is in the
Muckenhoupt $A_{p}$ class if it satisfies the condition
$$[\omega]_{A_{p}}:=\sup_{Q}\bigg(\frac{1}{|Q|}\int_{Q}\omega(x)dx\bigg)\bigg(\frac{1}{|Q|}\int_{Q}\omega(x)^{-\frac{1}{p-1}}dx\bigg)^{p-1}<\infty.$$
And a weight function $\omega$ belongs to the class $A_{1}$ if there exists $C> 0$ such that for every cube Q,
$$\frac{1}{|Q|}\int_{Q}\omega(x)dx\leq C\mathop\mathrm{ess~inf}_{x\in Q}\omega(x),$$
and the infimum of $C$ is denoted by $[\omega]_{A_{1}}$. A weight $\omega$ is called an $A_{\infty}$ weight if
$$[\omega]_{A_{\infty}}:=\sup_{Q}\bigg(\frac{1}{|Q|}\int_{Q}\omega(x)dx\bigg)\exp\bigg(\frac{1}{|Q|}\int_{Q}\log \omega(x)^{-1}dx\bigg)<\infty.$$
In fact, $A_{\infty}=\bigcup_{1\leq p<\infty}A_{p}$.

Weighted inequalities arise naturally in Fourier analysis, but their use is best justified by the variety of applications in which they appear. For example, the theory of weights plays an important role in the study of boundary value problems for Laplace's equation on Lipschitz domains. Other applications of weighted inequalities include vector-valued inequalities, extrapolation of operators and applications to certain classes of integral equation and nonlinear partial differential equation. There are a number of classical results demonstrate that the Muckenhoupt $A_{p}$ classes are
the right collections of weights to do harmonic analysis on weighted spaces. The main results along these
lines are the equivalence between the $\omega\in A_{p}$ condition and the $L^{p}(\omega)$ boundedness (or weak boundedness) of maximal operator and singular integral operators.

A well known result of Muckenhoupt \cite{M} showed that the Hardy-Littlewood maximal function
$$Mf(x)=\sup_{Q\ni x}\frac{1}{|Q|}\int_{Q}|f(y)|dy$$
is (weak) bounded on weighted Lebesgue spaces $L^{p}(\omega)$ if and only if $\omega\in A_{p}$ for $1<p<\infty$ (for the case $n=1$). Hunt, Muckenhoupt and Wheeden \cite{HMW} proved that the $A_{p}$ condition also characterizes the $L^{p}(\omega)$ boundedness of the Hilbert transform
$$Hf(x)=\frac{1}{\pi} \mathrm{p.v.}\int_{\mathbb{R}}\frac{f(y)}{x-y}dy.$$
Later, Coifman and Fefferman \cite{CF} extended the $A_{p}$ theory to the case $n\geq1$ and general Calder\'on-Zygmund operators, they also proved that $A_{p}$ weights satisfy the crucial reverse H\"{o}lder condition.

In 2009, Komori and Shirai \cite{KS} introduced the weighted Morrey spaces. Let $0<q<p<\infty$, $\omega$ be a weight and $\omega(Q):=\int_{Q}\omega(x)dx$. Then a weighted Morrey space is defined by
$$M^{p}_{q}(\omega)=\bigg\{f\in L^{q}_{loc}(\omega):\|f\|_{M^{p}_{q}(\omega)}:=\sup_{Q}\frac{1}{\omega(Q)^{1/q-1/p}}\bigg(\int_{Q}|f(x)|^{q}\omega(x)dx\bigg)^{1/q}<\infty\bigg\},$$
and a weighted weak Morrey space is defined by
$$WM^{p}_{q}(\omega)=\Big\{f\in L^{q}_{loc}(\omega):\|f\|_{WM^{p}_{q}(\omega)}<\infty\Big\},$$
where
$$\|f\|_{WM^{p}_{q}(\omega)}:=\sup_{Q}\frac{1}{\omega(Q)^{1/q-1/p}}\sup_{\lambda>0}\lambda\Big(\int_{\{x\in Q: |f(x)|>\lambda\}}\omega(x)dx\Big)^{1/q}.$$
\vspace{0.2cm}
They proved that if $\omega\in A_{p}$, then $M$ is bounded on $M^{p}_{q}(\omega)$. An interesting question is raised. Is $\omega$ in $A_{p}$ if $M$ is bounded on $M^{p}_{q}(\omega)$ for $1<q<p<\infty$? We will give an affirmative answer as follows.

\vspace{0.5cm}

\begin{theorem}\label{main1.1}
Let $1<q<p<\infty.$ The following statements are equivalent:
\begin{enumerate}
\item [\rm(1)] $\omega\in A_{p}$;
\item [\rm(2)] $M$ is a bounded operator from $L^{p}(\omega)$ to $L^{p,\infty}(\omega)$;
\item [\rm(3)] $M$ is a bounded operator from $L^{p}(\omega)$ to $M^{p}_{q}(\omega)$;
\item [\rm(4)] $M$ is a bounded operator from $L^{p}(\omega)$ to $WM^{p}_{q}(\omega)$;
\item [\rm(5)] $M$ is a bounded operator from $M_{q}^{p}(\omega)$ to $M_{q}^{p}(\omega)$;
\item [\rm(6)] $M$ is a bounded operator from $M_{q}^{p}(\omega)$ to $WM_{q}^{p}(\omega)$.
\end{enumerate}
\end{theorem}

\begin{remark}
It should be point out that the main contribution of this paper in Theorem \ref{main1.1} is $(2)\Rightarrow (3)$ and $(4)\Rightarrow (1)$, and other implications have been showed in \cite{CF} and \cite{KS} or follow from the trivial embedding properties.
\end{remark}

\vspace{0.3cm}

For the case $p=1$, we have

\begin{theorem}\label{main1.2}
Let $0<q<1.$ The following statements are equivalent:
\begin{enumerate}
\item [\rm(1)] $\omega\in A_{1}$;
\item [\rm(2)] $M$ is a bounded operator from $L^{1}(\omega)$ to $L^{1,\infty}(\omega)$;
\item [\rm(3)] $M$ is a bounded operator from $L^{1}(\omega)$ to $M^{1}_{q}(\omega)$;
\item [\rm(4)] $M$ is a bounded operator from $L^{1}(\omega)$ to $WM^{1}_{q}(\omega)$.
\end{enumerate}
\end{theorem}

\vspace{0.3cm}

For the $n$th dimensional case, the $A_{p}$ condition also characterizes the $L^{p}(\omega)$ boundedness of the Riesz transform
$$R_{j}f(x)=c_{n} \mathrm{p.v.}\int_{\mathbb{R}^n}\frac{x_{j}-y_{j}}{|x-y|^{n+1}}f(y)dy, \quad j=1,2,\cdots,n,$$
where $c_{n}=\Gamma(\frac{n+1}{2})/\pi^{\frac{n+1}{2}}$. In this paper, we will show that the $A_{p}$ condition is also necessary for the boundedness of the Riesz transforms on weighted Morrey spaces.
\begin{theorem}\label{main1.3}
Let $1<q<p<\infty.$ The following statements are equivalent:
\begin{enumerate}
\item [\rm(1)] $\omega\in A_{p}$;
\item [\rm(2)] $R_{j}$ is a bounded operator from $L^{p}(\omega)$ to $L^{p}(\omega)$ for all $j=1,2,\cdots,n$;
\item [\rm(3)] $R_{j}$ is a bounded operator from $L^{p}(\omega)$ to $M^{p}_{q}(\omega)$ for all $j=1,2,\cdots,n$;
\item [\rm(4)] $R_{j}$ is a bounded operator from $L^{p}(\omega)$ to $WM^{p}_{q}(\omega)$ for all $j=1,2,\cdots,n$;
\item [\rm(5)] $R_{j}$ is a bounded operator from $M_{q}^{p}(\omega)$ to $M_{q}^{p}(\omega)$ for all $j=1,2,\cdots,n$;
\item [\rm(6)] $R_{j}$ is a bounded operator from $M_{q}^{p}(\omega)$ to $WM_{q}^{p}(\omega)$ for all $j=1,2,\cdots,n$.
\end{enumerate}
\end{theorem}

\vspace{0.5cm}

For the case $p=1$, we have

\begin{theorem}\label{main1.4}
Let $0<q<1.$ The following statements are equivalent:
\begin{enumerate}
\item [\rm(1)] $\omega\in A_{1}$;
\item [\rm(2)] $R_{j}$ is a bounded operator from $L^{1}(\omega)$ to $L^{1,\infty}(\omega)$ for all $j=1,2,\cdots,n$;
\item [\rm(3)] $R_{j}$ is a bounded operator from $L^{1}(\omega)$ to $M^{1}_{q}(\omega)$ for all $j=1,2,\cdots,n$;
\item [\rm(4)] $R_{j}$ is a bounded operator from $L^{1}(\omega)$ to $WM^{1}_{q}(\omega)$ for all $j=1,2,\cdots,n$.
\end{enumerate}
\end{theorem}

\vspace{0.3cm}

In \cite[page 198]{S}, Stein showed that the convolution operators $T_{\epsilon}$ is bounded on weighted Lebesgue spaces if and only if $\omega\in A_{p}$. The convolution operators $T_{\epsilon}$ are defined by
$$T_{\epsilon}f=f\ast \Phi_{\epsilon},$$
where $\Phi$ is nonnegative, radial, and (radially) decreasing, with $\int \Phi(x)dx=1$, and we define $\Phi_{\epsilon}(x)=\epsilon^{-n}\Phi(x/\epsilon)$. We designate the class of such $\Phi$ by $\mathcal{R}$. Notice that if $\Phi$ is in $\mathcal{R}$ , then so is $\Phi_{\epsilon}$. It is useful to recall that
\begin{equation*}
Mf(x)=\sup_{\Phi\in \mathcal{R}}|f|\ast \Phi(x).
\end{equation*}
In fact, if $B_{\epsilon}=\{x:|x|<\epsilon\}$ then $|B_{\epsilon}|^{-1}\chi_{B_{\epsilon}}\in \mathcal{R}$, and the supremum over these elements of $\mathcal{R}$ is, by definition, equal to $Mf(x)$. For the other direction, we need only recall that any element of $\mathcal{R}$ is a limit of weighted averages of the $|B_{\epsilon}|^{-1}\chi_{B_{\epsilon}}$ (see \cite[Chapter 2, \S 2.1]{S}).

\begin{theorem}\label{main1.5}
Let $1<q<p<\infty.$ The following statements are equivalent:
\begin{enumerate}
\item [\rm(1)] $\omega\in A_{p}$;
\item [\rm(2)] $T_{\epsilon}$ is a bounded operator from $L^{p}(\omega)$ to $L^{p}(\omega)$;
\item [\rm(3)] $T_{\epsilon}$ is a bounded operator from $L^{p}(\omega)$ to $L^{p,\infty}(\omega)$;
\item [\rm(4)] $T_{\epsilon}$ is a bounded operator from $L^{p}(\omega)$ to $M^{p}_{q}(\omega)$;
\item [\rm(5)] $T_{\epsilon}$ is a bounded operator from $L^{p}(\omega)$ to $WM^{p}_{q}(\omega)$;
\item [\rm(6)] $T_{\epsilon}$ is a bounded operator from $M_{q}^{p}(\omega)$ to $M_{q}^{p}(\omega)$;
\item [\rm(7)] $T_{\epsilon}$ is a bounded operator from $M_{q}^{p}(\omega)$ to $WM_{q}^{p}(\omega)$.
\end{enumerate}
\end{theorem}

Finally, we also consider Muckenhoupt weight theory related to weighted BMO spaces. Weighted \rm{BMO} spaces play a fundamental role in many fields of mathematics such as harmonic analysis and partial differential equations; see \cite{BPT}, \cite{BT}, \cite{G} and \cite{JM}. Let us introduce the weighted \rm{BMO} spaces.

Let $1\leq p<\infty$. Given a a nonnegative locally integrable function $\omega$, the weighted $\mathrm{BMO}$ space $\mathrm{BMO}^{p}(\omega)$ is defined be the set of all functions $f\in L^{1}_{\mathrm{loc}}(\mathbb{R}^{n})$ such that
$$\|f\|_{\mathrm{BMO}^{p}(w)}:=\sup_{Q}\bigg(\frac{1}{w(Q)}\int_{Q}|f(y)-f_{Q}|^{p}\omega(y)^{1-p}dy\bigg)^{1/p}<\infty,$$
where the supremum is taken over all cubes $Q\subset \mathbb{R}^{n}$. We write $\mathrm{BMO}^{1}(\omega)=\mathrm{BMO}(\omega)$ simple.

\begin{remark}
\rm(i) For $1\leq p<\infty$ and $\omega\in A_{1}$, Garc\'{i}a-Cuerva \cite{G} proved that $\mathrm{BMO}(\omega)=\mathrm{BMO}^{p}(\omega)$ with equivalence of the corresponding norms.

\rm(ii) For the case $p=\infty$, the weighted $\mathrm{BMO}$ space can be defined as
$$\|f\|_{\mathrm{BMO}^{\infty}(\omega)}:=\mathop\mathrm{ess~ sup}_{x\in Q}\frac{|f(x)-f_{Q}|}{\omega(x)}.$$
\end{remark}

Now we state the result of the characterization of Muckenhoupt $A_{p}$ class via weighted \rm{BMO} spaces.
\begin{theorem}\label{main1.6}
Let $1\leq p< \infty$ and $\frac{1}{p} + \frac{1}{p'}= 1$. Given a nonnegative locally integrable function $\omega$, $\omega\in A_{p}$ if and only if $\omega\in \mathrm{BMO}^{p'}(\omega)$. Moreover,
$$[\omega]^{1/p}_{A_{p}}\leq \|\omega\|_{{\rm BMO}^{p'}(\omega)}+1\leq 3[\omega]^{1/p}_{A_{p}}.$$
\end{theorem}

\vspace{0.3cm}

For the case $p=\infty$, we have the following result.
\begin{theorem}\label{main1.7}
Given a nonnegative locally integrable function $\omega\in A_{\infty}$, then $\omega\in \mathrm{BMO}(\omega)$.
\end{theorem}

Though the proof of Theorem \ref{main1.6} is fairly straightforward, this result seems interesting. It appears to provide some insight into the growth allowed for Muckenhoupt weights. As an application, the operator norms of weighted Hardy-Littlewood average operator $U_{\psi}$ on weighted \rm{BMO} spaces are also obtained.

\vspace{0.3cm}

\begin{theorem}\label{main1.8}
Let $\psi: [0,1]\rightarrow [0,\infty)$ be a function, $ p\in [1, \infty]$, $\alpha\in \mathbb{R}$, $\omega\in A_{p'}$ and $\omega(tx)=t^{\alpha}\omega(x)$ for all $t\in (0,\infty)$.
Then $U_{\psi}: \mathrm{BMO}^{p}(\omega)\rightarrow \mathrm{BMO}^{p}(\omega)$ exists as a bounded operator if and only if
\begin{equation}\label{1.5-1}
 \int_{0}^{1}t^{\alpha}\psi(t)dt<\infty.
\end{equation}
Moreover, when $(\ref{1.5-1})$ holds, the operator norm of $U_{\psi}$ on $\mathrm{BMO}^{p}(\omega)$ is given by
$$\|U_{\psi}\|_{\mathrm{BMO}^{p}(\omega)\rightarrow \mathrm{BMO}^{p}(\omega)}=\int_{0}^{1}t^{\alpha}\psi(t)dt.$$
\end{theorem}

\vspace{0.3cm}
\begin{remark}
$(\mathrm{i})$ If $\omega(x)=|x|^{\alpha}$, the condition $\omega(tx)=t^{\alpha}\omega(x)$ holds.

$(\mathrm{ii})$ The norm from $\mathrm{BMO}^{p}(\omega)$ to itself of the operator $U_{\psi}$ is independent of $p$.
\end{remark}

\vspace{0.3cm}

The definition of weighted Hardy-Littlewood average operator $U_{\psi}$ as follows.

Let $\psi: [0,1]\rightarrow [0,\infty)$ be a function. For a measurable complex valued function $f$ on $\mathbb{R}^{n}$, the weighted Hardy-Littlewood average operator $U_{\psi}$ is defined as
$$(U_{\psi}f)(x)=\int_{0}^{1}f(tx)\psi(t)dt.$$
It was first defined by Carton-Lebrun and Fosset in \cite{CF2} and they showed that if $t^{1-n}\psi(t)$ is bounded on $[0,1]$ then $U_{\psi}$ is bounded on $\mathrm{BMO}(\mathbb{R}^{n})$. In 2001, Xiao \cite{X} obtained that $U_{\psi}$ is bounded on $L^{p}(\mathbb{R}^{n})$, $1\leq p\leq \infty$, if and only if
\begin{equation}\label{1.5-2}
\int_{0}^{1}t^{-n/p}\psi(t)dt<\infty.
\end{equation}
Meanwhile, when $(\ref{1.5-2})$ holds,
$$\|U_{\psi}\|_{L^{p}(\mathbb{R}^{n})\rightarrow L^{p}(\mathbb{R}^{n})}=\int_{0}^{1}t^{-n/p}\psi(t)dt.$$
Xiao also showed that $U_{\psi}$ is bounded on $\mathrm{BMO}(\mathbb{R}^{n})$  if and only if
\begin{equation}\label{1.5-3}
\int_{0}^{1}\psi(t)dt<\infty.
\end{equation}
And when $(\ref{1.5-3})$ holds, the precise norm of $U_{\psi}$ on $\mathrm{BMO}(\mathbb{R}^{n})$ is given by
$$\|U_{\psi}\|_{\mathrm{BMO}(\mathbb{R}^{n})\rightarrow \mathrm{BMO}(\mathbb{R}^{n})}=\int_{0}^{1}\psi(t)dt.$$

\vspace{0.5cm}

This paper is organized as follows. In Section 2, we prove that $A_{p}$ weights give a characterization of weighted Morrey spaces $M^{p}_{q}(\omega)$ boundedness for the Hardy-Littlewood maximal operator (see Theorem \ref{main1.1} and Theorem \ref{main1.2}), Riesz transform (see Theorem \ref{main1.3} and Theorem \ref{main1.4}) and convolution operators $T_{\epsilon}$ (see Theorem \ref{main1.5}). In Section 3, we address $A_{p}$ weight theory in connection to weighted {\rm BMO} spaces and its application (see Theorem \ref{main1.6}, Theorem \ref{main1.7} and Theorem \ref{main1.8}).

Throughout this paper, the letter $C$ denotes constants which are independent of main variables and may change from one occurrence to another.

\vspace{0.3cm}

\section{Proof of Theorem \ref{main1.1} $\sim$ Theorem \ref{main1.5}}

First of all, we compare with weighted Morrey spaces and weighted weak Morrey spaces. It is clear that $M^{p}_{r}(\omega)$ is contained in $WM_{q}^{p}(\omega)$ and $\|\cdot\|_{WM^{p}_{q}(\omega)}\leq \|\cdot\|_{M^{p}_{q}(\omega)} \leq \|\cdot\|_{M^{p}_{r}(\omega)}$ if $1< q\leq r< p<\infty$.
However, for $1< r< q\leq p<\infty$, one has the reverse inequality as follows.

\vspace{0.3cm}
\begin{propo}\label{prop2.1}
Let $0< q\leq p<\infty$ and $\omega$ be a nonnegative locally integral function on $\mathbb{R}^{n}$.
\begin{enumerate}
\item [\rm(1)] If $f\in WM_{q}^{p}(\omega)$, then for all $r\in(0,q)$ and all cube $Q$, there exists a constant $C>0$ such that
$$\int_{Q}|f(x)|^{r}\omega(x)dx\leq C\omega(Q)^{1-r/p}\|f\|^{r}_{WM^{p}_{q}(\omega)}.$$
\item [\rm(2)] If there exists $r\in(0,q)$ and constant $C>0$ such that
\begin{equation}\label{2-1}
\int_{\Omega}|f(x)|^{r}\omega(x)dx\leq C\omega(\Omega)^{1-r/p}
\end{equation}
for all set $\Omega$ with $\omega(\Omega)<\infty$, then $f\in WM^{p}_{q}(\omega)$.
\end{enumerate}
\end{propo}
\begin{proof}
Let $f\in WM^{p}_{q}(\omega)$. For any $\lambda>0$
$$\lambda\Big(\int_{\{x\in Q: |f(x)|>\lambda\}}\omega(x)dx\Big)^{1/q}\leq \|f\|_{WM_{q}^{p}(\omega)}\omega(Q)^{1/q-1/p}.$$
Choose
$$N=\|f\|_{WM^{p}_{q}(\omega)}\big(\frac{r}{q-r}\big)^{1/q}\omega(Q)^{-1/p},$$
for any $0<r<q$,
\begin{eqnarray*}
\int_{Q}|f(x)|^{r}\omega(x)dx&=&r\int_{0}^{\infty}\lambda^{r-1}\omega\big(\{x\in Q:|f(x)|>\lambda\}\big)d\lambda\\
&\leq&r\int_{0}^{N}\lambda^{r-1}\omega(Q)d\lambda+r\int_{N}^{\infty}\lambda^{r-1}\frac{\|f\|^{q}_{WM^{p}_{q}(\omega)}}{\lambda^{q}}\omega(Q)^{1-q/p}d\lambda\\
&=&N^{r}\omega(Q)+\|f\|^{p}_{WM^{p}_{q}(\omega)}\frac{r}{q-r}N^{r-q}\omega(Q)^{1-q/p},
\end{eqnarray*}
which implies
$$\Big(\int_{Q}|f(x)|^{r}\omega(x)dx\Big)^{1/r}\leq 2\|f\|_{WM^{p}_{q}(\omega)}\big(\frac{r}{q-r}\big)^{1/q}\omega(Q)^{1/r-1/p};$$
that is,
$$\int_{Q}|f(x)|^{r}\omega(x)dx\leq C\omega(Q)^{1-r/p}\|f\|^{r}_{WM^{p}_{q}(\omega)}.$$
\vskip 0.5cm
\noindent
{\it Proof of \rm{(2)}.} For any $\lambda>0$, take $E=\{x\in \mathbb{R}^{n}:|f(x)|>\lambda\}$, then $\omega(E)<\infty$. Otherwise, there is a sequence $\{E_{k}\}$ of measurable sets such that $E_{k}\subset E$ and $\omega(E_{k})=k$ for $k=0,1,2,\cdots$. Thus for every $k$,
$$\lambda^{r}k=\lambda^{r}\omega(E_{k})\leq \int_{E_{k}}|f(x)|^{r}\omega(x)dx\leq C\omega(E_{k})^{1-r/p}=Ck^{1-r/p}.$$
However, it is not true.

Thus, for any cube $Q\subset \mathbb{R}^{n}$, take $\Omega=\{x\in Q: |f(x)|>\lambda\}$(for $p=q$, take $\Omega=E$), then $\omega(\Omega)\leq \omega(E)<\infty$. By (\ref{2-1}), we have
$$\lambda^{r}\omega(\Omega)\leq \int_{\Omega}|f(x)|^{r}\omega(x)dx\leq C\omega(\Omega)^{1-r/p}.$$
It follows that
$$\lambda\omega(Q)^{1/p-1/q}\omega(\Omega)^{1/q}\leq \lambda \omega(\Omega)^{1/p}\leq C.$$
Hence
$$\sup_{Q}\frac{1}{\omega(Q)^{1/q-1/p}}\sup_{\lambda>0}\lambda\Big(\int_{\{x\in Q: |f(x)|>\lambda\}}\omega(x)dx\Big)^{1/q}\leq C,$$
then $f\in WM^{p}_{q}(\omega)$.
\end{proof}

\begin{remark}\label{rem}\rm
In fact, for $p=q$, we can replace "for all cube $Q$" in Proposition \ref{prop2.1} (1) with "for all set $\Omega$ with $\omega(\Omega)<\infty$". For the proof similar arguments are applied with necessary modifications.
\end{remark}

As a corollary of Proposition \ref{prop2.1}, we get

\begin{corollary}\label{cor2.1}
Let $0<r<q<p<\infty$ and $\omega$ be a nonnegative locally integral function on $\mathbb{R}^{n}$. Weighted Morrey space $M^{p}_{r}(\omega)$ is contained in $WM_{q}^{p}(\omega)$ and
\begin{equation*}
\|\cdot\|_{M^{p}_{r}(\omega)}\leq C\|\cdot\|_{WM_{q}^{p}(\omega)}\leq C\|\cdot\|_{M^{p}_{q}(\omega)}\leq C\|\cdot\|_{WM^{p}_{p}(\omega)}=C\|\cdot\|_{L^{p,\infty}(\omega)}.
\end{equation*}
\end{corollary}
\vskip 0.5cm
\noindent

{\it Proof of Theorem \ref{main1.1}.}
The equivalence of $(1)$ and $(2)$ was proved in \cite{CF}. By Corollary \ref{cor2.1}, it is obvious that $(2)$ implies $(3)$, and $(3)$ implies $(4)$ follows from the trivial embedding properties.

$(4)\Rightarrow (1)$: Let $Q$ be any cube and assuming for a moment that $\int_{Q}\omega(x)^{1-p'}dx<\infty$. We take $f=\omega^{1-p'}\chi_{Q}$. For any $0<\lambda<\omega^{1-p'}(Q)/|Q|$, we obtain that for any $x\in Q$, $Mf(x)>\lambda$. Then
\begin{eqnarray*}
\frac{\lambda\omega(Q)^{1/p}}{|Q|^{1/p}}&=&\frac{\lambda}{|Q|^{1/p}}\frac{1}{\omega(Q)^{1/q-1/p}}\Big(\int_{Q}\omega(x)dx\Big)^{1/q}\\
&=&\frac{1}{|Q|^{1/p}}\frac{\lambda}{\omega(Q)^{1/q-1/p}}\Big(\int_{\{x\in Q: ~Mf(x)>\lambda\}}\omega(x)dx\Big)^{1/q}\\
&\leq&\frac{C}{|Q|^{1/p}}\Big(\int_{\mathbb{R}^{n}}[\omega(x)^{1-p'}\chi_{Q}(x)]^{p}\omega(x)dx\Big)^{1/p}\\
&=&\frac{C}{|Q|^{1/p}}\Big(\int_{Q}\omega(x)^{1-p'}dx\Big)^{1/p}.
\end{eqnarray*}
Hence by arbitrariness of $\lambda$, we get
\begin{equation*}
\frac{1}{|Q|}\int_{Q}\omega(x)dx\leq C\bigg(\frac{1}{|Q|}\int_{Q}\omega(x)^{1-p'}dx\bigg)^{1-p};
\end{equation*}
that is,
\begin{equation}\label{2-2}
\Big(\frac{1}{|Q|}\int_{Q}\omega(x)dx\Big)\Big(\frac{1}{|Q|}\int_{Q}\omega(x)^{1-p'}dx\Big)^{p-1}\leq C.
\end{equation}
If $\omega(x)^{1-p'}$ is not locally integrable, then we take $f=(\omega+\epsilon)^{1-p'}\chi_{Q}$ and $0<\lambda<(\omega+\epsilon)^{1-p'}(Q)/|Q|$ to get
$$\Big(\frac{1}{|Q|}\int_{Q}\omega(x)dx\Big)\Big(\frac{1}{|Q|}\int_{Q}(\omega(x)+\epsilon)^{-p'}\omega(x)dx\Big)^{p-1}\leq C,$$
for all $\epsilon>0$, from which we can still deduce (\ref{2-2}) by letting $\epsilon\rightarrow 0$. So $\omega\in A_{p}$.

In addition, $(1)\Rightarrow (5)$ was proved in \cite{KS}, and it is obvious that $(5)\Rightarrow (6)$. Since $\|\cdot\|_{M^{p}_{q}(\omega)}\leq \|\cdot\|_{L^{p}(\omega)}$ for any weight function $\omega$ and $1<q<p<\infty$, we have $(6)\Rightarrow (4)$. Finally, $(4)\Rightarrow (1)$ can be found above. \qed

\vskip 0.3cm
\noindent

{\it Proof of Theorem \ref{main1.2}.} We only give the proof of $(4)\Rightarrow (1)$, and other implications have been showed in [6], Corollary \ref{cor2.1} or follow from the trivial embedding properties.

$(4)\Rightarrow (1)$: Let $Q$ be any cube and $Q_{1}\subset Q$. If denote $f=\chi_{Q_{1}}$, then for any $0<\lambda<\frac{|Q_{1}|}{|Q|}$, we have
$$Q=\big\{x\in Q:Mf(x)>\lambda\big\}.$$
Applying $M$ is bounded from $L^{1}(\omega)$ to $WM^{1}_{q}(\omega)$, we get
$$\lambda\int_{Q}\omega(x)dx=\lambda\omega(Q)^{1-1/q}\Big(\int_{\{x\in Q: Mf(x)>\lambda\}}\omega(x)dx\Big)^{1/q}\leq C\int_{Q_{1}}\omega(x)dx.$$
Choose $\lambda=\frac{|Q_{1}|}{2|Q|}$, then
$$\frac{1}{|Q|}\int_{Q}\omega(x)dx\leq \frac{C}{|Q_{1}|}\int_{Q_{1}}\omega(x)dx.$$
It follows from Lebesgue differentiation theorem that
$$\frac{1}{|Q|}\int_{Q}\omega(x)dx\leq C\omega(x)\quad a.e.\quad x\in Q,$$
which proves
$$M\omega(x)\leq C\omega(x) \quad a.e. \quad x\in \mathbb{R}^n.$$
Thus $\omega\in A_{1}$. \qed

\vspace{0.3cm}

Sharp bounds for the operators norms in terms of the $A_{p}$ constants of the weights have been investigated as well. Buckley \cite{B} obtained that for $1<p<\infty$,
$$\|M\|_{L^{p}(\omega)\rightarrow L^{p}(\omega)}\leq C[\omega]_{A_{p}}^{1/(p-1)}$$
and the power $[\omega]_{A_{p}}^{1/(p-1)}$ is the best possible. The weak type bound was found and shown to be best possible by
Muckenhoupt \cite{M}, when the proof is examined closely; that is,
\begin{equation}\label{2-3}
\|M\|_{L^{p}(\omega)\rightarrow L^{p,\infty}(\omega)}\leq C[\omega]_{A_{p}}^{1/p}
\end{equation}
and the power $[\omega]_{A_{p}}^{1/p}$ is the best possible. However, only partial results of singular integrals operators are known. The interest in sharp weighted norm for singular integral operators is motivated in part by applications in partial differential
equations. We refer the reader to \cite{ATE}, \cite{DGPP}, \cite{LMP}, \cite{P1}, \cite{P2} and \cite{P3}.

Now, we give the sharp estimate for the boundedness of $M$ from weighted Lebesgue spaces to weighted (weak) Morrey spaces.
\begin{theorem}\label{thm1}
If $1<q<p<\infty$ and $\omega\in A_{p}$, then $\|Mf\|_{M^{p}_{q}(\omega)}\leq C[\omega]_{A_{p}}^{1/p}\|f\|_{L^{p}(\omega)}$ and the power $[\omega]_{A_{p}}^{1/p}$ is best possible.
\end{theorem}
\begin{theorem}\label{thm2}
If $1<q<p<\infty$ and $\omega\in A_{p}$, then $\|Mf\|_{WM^{p}_{q}(\omega)}\leq C[\omega]_{A_{p}}^{1/p}\|f\|_{L^{p}(\omega)}$ and the power $[\omega]_{A_{p}}^{1/p}$ is best possible.
\end{theorem}
We focus on the proof of Theorem \ref{thm1}. For the proof of Theorem \ref{thm2} similar arguments are applied
with necessary modifications, we omit the details.

\vspace{0.3cm}

{\it Proof of Theorem \ref{thm1}.}
On the one hand, by Corollary \ref{cor2.1} and the inequalities (\ref{2-3}), we have
$$\|Mf\|_{M^{p}_{q}(\omega)}\leq C\|Mf\|_{L^{p,\infty}(\omega)}\leq C[\omega]^{1/p}_{A_{p}}\|f\|_{L^{p}(\omega)}.$$
On the other hand, let $Q$ be any cube, take $f=\omega^{1-p'}\chi_{Q}$ and $\lambda=\frac{\omega^{1-p'}(Q)}{2|Q|}$ (If $\omega(x)^{1-p'}$ is not locally integrable, then we take $f=(\omega+\epsilon)^{-p'/p}\chi_{Q}$ and $\lambda=\frac{(\omega+\epsilon)^{1-p'}(Q)}{2|Q|}$ ). Since
$$\{x: x\in Q\}=\{x\in Q: ~Mf(x)>\lambda\},$$
then
\begin{eqnarray*}
\frac{\lambda\omega(Q)^{1/p}}{|Q|^{1/p}}&=&\frac{\lambda}{|Q|^{1/p}}\frac{1}{\omega(Q)^{1/q-1/p}}\Big(\int_{Q}\omega(x)dx\Big)^{1/q}\\
&=&\frac{1}{|Q|^{1/p}}\frac{\lambda}{\omega(Q)^{1/q-1/p}}\Big(\int_{\{x\in Q: ~Mf(x)>\lambda\}}\omega(x)dx\Big)^{1/q}\\
&\leq&\frac{1}{|Q|^{1/p}}\frac{1}{\omega(Q)^{1/q-1/p}}\Big(\int_{Q}|Mf(x)|^{q}\omega(x)dx\Big)^{1/q}\\
&\leq&\frac{\|M\|_{L^{p}(\omega)\rightarrow M^{p}_{q}(\omega)}}{|Q|^{1/p}}\Big(\int_{\mathbb{R}^{n}}[\omega(x)^{1-p'}\chi_{Q}(x)]^{p}\omega(x)dx\Big)^{1/p}\\
&=&\frac{\|M\|_{L^{p}(\omega)\rightarrow M^{p}_{q}(\omega)}}{|Q|^{1/p}}\Big(\int_{Q}\omega(x)^{1-p'}dx\Big)^{1/p},
\end{eqnarray*}
which implies that
$$[\omega]_{A_{p}}\leq 2^{p}\|M\|^{p}_{L^{p}(\omega)\rightarrow M^{p}_{q}(\omega)}$$
and Theorem \ref{thm1} follows. \qed

\vskip 0.3cm
\noindent

Similar to Theorem \ref{main1.1} and Theorem \ref{main1.2}, we only prove $(4)\Rightarrow (1)$ of Theorem \ref{main1.3} and Theorem \ref{main1.4}. In fact, other implications have been showed in \cite{CF} and \cite{KS} or follow from the trivial embedding properties.

\vskip 0.3cm
\noindent

{\it Proof of Theorem \ref{main1.3} and Theorem \ref{main1.4}.} $(4)\Rightarrow (1)$: We prove the $n$th dimensional case when $n\geq 2$. The one-dimensional case is obtained by a simple adaptation of the following argument.

Let $Q$ be a cube and let $f$ be a nonnegative function on $\mathbb{R}^n$ supported in $Q$ that satisfies $f_{Q}>0$. Let $Q'$ be the cube that shares a corner with $Q$, which has the same length as $Q$ with $x_{j}\geq y_{j}$ for all $x\in Q'$ and $y\in Q$. Then for $x\in Q'$ and $y\in Q$, we have $\sum_{j=1}^{n}x_{j}-y_{j}\geq |x-y|$ and $|x-y|^{-n}\geq C|Q|^{-1}$, which implies that for $x\in Q'$,
$$\Big|\sum_{j=1}^{n}R_{j}f(x)\Big|=c_{n}\Big|\sum_{j=1}^{n}\int_{\mathbb{R}^n}\frac{x_{j}-y_{j}}{|x-y|^{n+1}}f(y)dy\Big|\geq C\int_{Q}\frac{f(y)}{|x-y|^n}\geq Cf_{Q}.$$
For all $0<\lambda<Cf_{Q}$ we have
$$Q'=\big\{x\in Q':\Big|\sum_{j=1}^{n}R_{j}f(x)\Big|>\lambda\big\}.$$

Since the operator $\sum_{j=1}^{n}R_{j}$ is bounded from $L^{p}(\omega)$ to $WM^{p}_{q}(\omega)$ for $p\geq 1$, then
$$\omega(Q')\leq \frac{C}{\lambda^{p}}\int_{Q}f(x)^{p}\omega(x) dx,$$
for all $0<\lambda<Cf_{Q}$, it proves that
\begin{equation}\label{1.3.1}
f_{Q}^{p}\leq \frac{C}{\omega(Q')}\int_{Q}f(x)^{p}\omega(x)dx.
\end{equation}
We observe that we can reverse the roles of $Q$ and $Q'$ and obtain
\begin{equation}\label{1.3.2}
g_{Q'}^{p}\leq \frac{C}{\omega(Q)}\int_{Q'}g(x)^{p}\omega(x)dx,
\end{equation}
for all $g$ supported in $Q'$. Taking $g=\chi_{Q'}$ in (\ref{1.3.2}), which gives that
$$\omega(Q)\leq C\omega(Q').$$
Using this estimate and (\ref{1.3.2}), we obtain
\begin{equation}\label{1.3.3}
f_{Q}^{p}\leq \frac{C}{\omega(Q)}\int_{Q}f(x)^{p}\omega(x)dx.
\end{equation}

{\bf Case 1:} $p>1$. Taking $f=(\omega+\epsilon)^{-p'/p}\chi_{Q}$ in (\ref{1.3.3}), we have
$$\Big(\frac{1}{|Q|}\int_{Q}\omega(x)dx\Big)\Big(\frac{1}{|Q|}\int_{Q}(\omega(x)+\epsilon)^{1-p'}dx\Big)^{p}
\Big(\frac{1}{|Q|}\int_{Q}(\omega(x)+\epsilon)^{-p'}\omega(x)dx\Big)^{-1}\leq C.$$
Letting $\epsilon\rightarrow 0$ and it follows that $\omega\in A_{p}$.

{\bf Case 2:} $p=1$. For any $Q_{1}\subset Q$ and denote $f=\chi_{Q_{1}}$. By (\ref{1.3.3}),
$$\frac{1}{|Q|}\int_{Q}\omega(x)dx\leq \frac{C}{|Q_{1}|}\int_{Q_{1}}\omega(x)dx.$$
Applying Lebesgue differentiation theorem that
$$\frac{1}{|Q|}\int_{Q}\omega(x)dx\leq C\omega(x)\quad a.e.\quad x\in Q.$$
Thus $\omega\in A_{1}$. \qed

\vskip 0.3cm
\noindent

{\it Proof of Theorem \ref{main1.5}.} We need only to show that $(5)\Rightarrow (1)$. In fact, $(1)\Rightarrow (2)$ has been proved in \cite[page 198]{S} and $(1)\Rightarrow (6)$ can be obtained by the fact that $|T_{\epsilon}f|\leq Mf$(see \cite[page 57]{S}) and the boundedness of $M$ on weighted Morrey spaces,  other implications follow from the trivial embedding properties.

$(5)\Rightarrow (1)$: We first assume that $\Phi_{\epsilon}=|B_{\epsilon}|^{-1}\chi_{B_{\epsilon}}$. Let $B:=B(x_{0},\delta)$ be any ball and let $f\geq 0$ supported in $B$. If $x\in B$, then $B\subset R:=B(x,2\delta)$. Thus if $x\in B$ and $\epsilon=2\delta$, we have
\begin{eqnarray*}
T_{\epsilon}f(x)=f\ast\Phi_{\epsilon}(x)&=&\int_{\mathbb{R}^n}f(y)\Phi_{\epsilon}(x-y)dy\\
&=&\int_{\mathbb{R}^n}f(y)|B_{\epsilon}|^{-1}\chi_{B_{\epsilon}}(x-y)dy\\
&=&|B_{\epsilon}|^{-1}\int_{|x-y|<2\delta}f(y)dy\\
&\geq&2^{-n}\frac{1}{|B|}\int_{B}f(y)dy=2^{-n}f_{B}.
\end{eqnarray*}
For any $0<\lambda<2^{-n}f_{B}$, by the fact that $T_{\epsilon}$ is a bounded operator from $L^{p}(\omega)$ to $WM^{p}_{q}(\omega)$, we obtain
\begin{eqnarray*}
C\bigg(\int_{\mathbb{R}^n}f(x)^{p}\omega(x)dx\bigg)^{1/p}&\geq& \omega(B)^{1/p-1/q}\sup_{\lambda>0}\lambda\bigg(\int_{\{x\in B: \ T_{\epsilon}f(x)>\lambda\}}\omega(x)dx\bigg)^{1/q}\\
&\geq&\omega(B)^{1/p}\cdot2^{-n}f_{B}.
\end{eqnarray*}
Taking $f=(\omega+\epsilon)^{-p'/p}\chi_{B}$, the above inequality implies that
$$\Big(\frac{1}{|Q|}\int_{Q}\omega(x)dx\Big)\Big(\frac{1}{|Q|}\int_{Q}(\omega(x)+\epsilon)^{1-p'}dx\Big)^{p}
\Big(\frac{1}{|Q|}\int_{Q}(\omega(x)+\epsilon)^{-p'}\omega(x)dx\Big)^{-1}\leq C.$$
Letting $\epsilon\rightarrow 0$ and it follows that $\omega\in A_{p}$.

For a general $\Phi\in \mathcal{R}$, there exist some positive constants $c_{1}$ and $c_{2}$ such that $\Phi(x)\geq c_{1}\chi_{B_{c_{2}}}(x)$ (see \cite[page 199]{S}), then the same conclusion follows.
\qed

\vspace{0.5cm}

\section{Proof of Theorem \ref{main1.6} $\sim$ Theorem \ref{main1.8}}

Below we shall turn to discuss the properties of weighted {\rm BMO} spaces.

\begin{propo}\label{prop3.1}
Let $1\leq p\leq \infty$, $1/p+1/p'=1$ and $\omega\in A_{p}$. If $f\in {\rm BMO}^{p'}(\omega)$, then $|f|\in {\rm BMO}^{p'}(\omega)$. Moreover, $$\big\||f|\big\|_{{\rm BMO}^{p'}(\omega)}\leq (1+[\omega]_{A_{p}}^{1/p})\|f\|_{{\rm BMO}^{p'}(\omega)}.$$
\end{propo}
\begin{proof}
Let $f\in {\rm BMO}^{p'}(\omega)$.

{\bf Case 1:} $p=1$. For any cube $x\in Q$,
\begin{eqnarray*}
\big||f(x)|-(|f|)_{Q}\big|&\leq& \big||f(x)|-f_{Q}\big|+\big|f_{Q}-(|f|)_{Q}\big|\\
&\leq&\big|f(x)-f_{Q}\big|+\frac{1}{|Q|}\int_{Q}|f(y)-f_{Q}|dy\\
&\leq&\big|f(x)-f_{Q}\big|+M\omega(x)\|f\|_{{\rm BMO}^{\infty}(\omega)}.
\end{eqnarray*}
It follows from $\omega\in A_{1}$ that
$$M\omega(x)\leq [\omega]_{A_{1}}\omega(x) \quad a.e. \quad x\in \mathbb{R}^n.$$
Then
$$\big\||f|\big\|_{{\rm BMO}^{\infty}(\omega)}\leq (1+[\omega]_{A_{1}})\|f\|_{{\rm BMO}^{\infty}(\omega)}.$$

{\bf Case 2:} $1< p<\infty$. Since $\omega\in A_{p}$, we have
\begin{eqnarray*}
\big||f(x)|-(|f|)_{Q}\big|&\leq& \big||f(x)|-|f_{Q}|\big|+\big||f_{Q}|-(|f|)_{Q}\big|\\
&\leq&\big|f(x)-f_{Q}\big|+\frac{1}{|Q|}\int_{Q}|f(y)-f_{Q}|dy\\
&\leq&\big|f(x)-f_{Q}\big|+\frac{1}{|Q|}\Big(\int_{Q}|f(y)-f_{Q}|^{p'}\omega(y)^{1-p'}dy\Big)^{1/p'}\omega(Q)^{1/p}\\
&\leq&\big|f(x)-f_{Q}\big|+\frac{\omega(Q)}{|Q|}\|f\|_{{\rm BMO}^{p'}(\omega)},
\end{eqnarray*}
which yields that
\begin{eqnarray*}
&&\bigg(\frac{1}{\omega(Q)}\int_{Q}\big||f(x)|-(|f|)_{Q}\big|^{p'}\omega(x)^{1-p'}dx\bigg)^{1/p'}\\
&&\leq \bigg(\frac{1}{\omega(Q)}\int_{Q}|f(x)-f_{Q}\big|^{p'}\omega(x)^{1-p'}dx\bigg)^{1/p'}\\
&&\quad+\bigg(\frac{1}{\omega(Q)}\int_{Q}\Big(\frac{\omega(Q)}{|Q|}\|f\|_{{\rm BMO}^{p}(\omega)}\Big)^{p'}\omega(x)^{1-p'}dx\bigg)^{1/p'}\\
&&\leq \|f\|_{{\rm BMO}^{p'}(\omega)}+\|f\|_{{\rm BMO}^{p'}(\omega)}\bigg(\Big(\frac{1}{|Q|}\int_{Q}\omega(x)dx\Big)\Big(\frac{1}{|Q|}\int_{Q}\omega(x)^{1-p'}dx\Big)^{p-1}\bigg)^{1/p}\\
&&\leq (1+[\omega]_{A_{p}}^{1/p})\|f\|_{{\rm BMO}^{p'}(\omega)}.
\end{eqnarray*}

{\bf Case 3:} $p=\infty$. For any cube $Q$,
\begin{eqnarray*}
\big||f(x)|-(|f|)_{Q}\big|&\leq& \big||f(x)|-|f_{Q}|\big|+\big||f_{Q}|-(|f|)_{Q}\big|\\
&\leq&\big|f(x)-f_{Q}\big|+\frac{1}{|Q|}\int_{Q}|f(y)-f_{Q}|dy\\
&\leq&\big|f(x)-f_{Q}\big|+\frac{\omega(Q)}{|Q|}\|f\|_{{\rm BMO}(\omega)}.
\end{eqnarray*}

It is easy to see that
$$\frac{1}{\omega(Q)}\int_{Q}\big||f(x)|-(|f|)_{Q}\big|dx\leq 2\|f\|_{{\rm BMO}(\omega)}.$$

Combining the estimates above, we conclude that
$$\big\||f|\big\|_{{\rm BMO}^{p'}(\omega)}\leq (1+[\omega]_{A_{p}}^{1/p})\|f\|_{{\rm BMO}^{p'}(\omega)}.$$
Thus, we complete the proof of Proposition \ref{prop3.1}.
\end{proof}

\vspace{0.3cm}

{\it Proof of Theorem \ref{main1.5}.}
On the one hand, we should prove that $\omega\in A_{p}$ implies $\omega\in \rm{BMO}^{p'}(\omega)$.

\textbf{Case 1}: $p=1$.
Given a cube $Q$, it follows from the $A_{1}$ condition that
$$|\omega(x)-\omega_{Q}|\leq \omega(x)+\omega_{Q}\leq \big(1+[\omega]_{A_{1}}\big)\omega(x), ~~a.e. ~x\in Q,$$
which gives
$$\mathop\mathrm{ess~sup}_{x\in Q}\frac{|\omega(x)-\omega_{Q}|}{w(x)}\leq 1+[\omega]_{A_{1}}.$$

\textbf{Case 2}: $1<p<\infty$. For any cube $Q$,  we use Minkowski¡¯s inequality to establish
\begin{eqnarray*}
&&\bigg(\frac{1}{\omega(Q)}\int_{Q}|\omega(x)-\omega_{Q}|^{p'}\omega(x)^{1-p'}dx\bigg)^{1/p'}\\
&&\leq \frac{1}{\omega(Q)^{1/p'}}\bigg(\int_{Q}\omega(x)^{p'}\omega(x)^{1-p'}dx\bigg)^{1/p'}
+\frac{\omega_{Q}}{\omega(Q)^{1/p'}}\bigg(\int_{Q}\omega(x)^{1-p'}dx\bigg)^{1/p'}\\
&&:=I_{1}+I_{2}.
\end{eqnarray*}
It is easy to see that $I_{1}=1$. For $I_{2}$, $\omega\in A_{p}$ implies that
\begin{eqnarray*}
I_{2}&=&\frac{1}{|Q|}\bigg(\int_{Q}\omega(x)dx\bigg)^{1-1/p'}\bigg(\int_{Q}\omega(x)^{1-p'}dx\bigg)^{1/p'}\\
&=&\bigg(\frac{1}{|Q|}\int_{Q}\omega(x)dx\bigg)^{1/p}\bigg(\frac{1}{|Q|}\int_{Q}\omega(x)^{1-p'}dx\bigg)^{1/p'}\\
&\leq&[\omega]^{1/p}_{A_{p}}.
\end{eqnarray*}
Therefore $\omega\in \mathrm{BMO}^{p'}(\omega)$ and $\|\omega\|_{\mathrm{BMO}^{p'}(\omega)}\leq 1+[\omega]^{1/p}_{A_{p}}$.

On the other hand, the condition $\omega\in A_{p}$ turns out to be necessary for the conclusion $\omega\in \mathrm{BMO}^{p'}(\omega)$ holds.

\textbf{Case 1}: $p=1$. Given a cube $Q$, it follows from $\omega\in \mathrm{BMO}^{\infty}(\omega)$ that
$$\omega_{Q}\leq |\omega(x)-\omega_{Q}|+\omega(x)\leq \big(1+\|\omega\|_{\mathrm{BMO}^{\infty}(\omega)}\big)\omega(x),  ~~a.e.~~x\in Q,$$
then, $M(\omega)(x)\leq \big(1+\|\omega\|_{\mathrm{BMO}^{\infty}(\omega)}\big)\omega(x),a.e. x\in \mathbb{R}^{n}$; that is, $\omega\in A_{1}$.

\textbf{Case 2}: $1<p< \infty$. For any fixed cube $Q$, using Minkowski's inequality,
\begin{eqnarray*}
\|\omega\|_{\mathrm{BMO}^{p'}(\omega)}&\geq&\bigg(\frac{1}{w(Q)}\int_{Q}|\omega(y)-\omega_{Q}|^{p'}\omega(y)^{1-p'}dy\bigg)^{1/p'}\\
&\geq&\bigg(\frac{1}{w(Q)}\int_{Q}\omega_{Q}^{p'}\omega(y)^{1-p'}dy\bigg)^{1/p'}- \bigg(\frac{1}{w(Q)}\int_{Q}\omega(y)^{p'}\omega(y)^{1-p'}dy\bigg)^{1/p'}\\
&=&\bigg(\frac{1}{|Q|}\int_{Q}\omega(x)dx\bigg)^{1/p}\bigg(\frac{1}{|Q|}\int_{Q}\omega(x)^{1-p'}dx\bigg)^{1/p'}-1.
\end{eqnarray*}
Thus, we conclude that $\omega\in A_{p}$ and $[\omega]_{A_{p}}\leq \big(\|\omega\|_{\mathrm{BMO}^{p'}(\omega)}+1\big)^{p}$.

Hence, it immediately follows that $\omega\in A_{p}$ if and only if $\omega\in \mathrm{BMO}^{p'}(\omega)$ with
$$[\omega]^{1/p}_{A_{p}}\leq \|\omega\|_{\mathrm{BMO}^{p'}(\omega)}+1\leq [\omega]^{1/p}_{A_{p}}+2.$$
It is easy to see that for $1\leq p<\infty$, $[\omega]_{A_{p}}\geq 1$. Thus, the proof of Theorem \ref{main1.5} is completed. \qed

\vskip 0.5cm
\noindent

{\it Proof of Theorem \ref{main1.7}.} For any $\omega\in A_{\infty}$, there exists $1\leq p_{0}<\infty$ such that $\omega\in A_{p_{0}}$. Applying Theorem \ref{main1.6}, we have
$$\omega\in \mathrm{BMO}^{p'_{0}}(\omega)\subset \mathrm{BMO}(\omega).$$
We finish the proof of Theorem \ref{main1.7}. \qed

\vspace{0.5cm}
\begin{remark}
For the case $p=\infty$. For any non-negative, locally integrable function $\omega$,
$$\|\omega\|_{\mathrm{BMO}(\omega)}=\sup_{Q}\frac{1}{\omega(Q)}\int_{Q}|\omega(x)-\omega_{Q}|dx\leq 2.$$
Then $\omega\in A_{\infty}$ is not necessary for the conclusion $\omega\in \mathrm{BMO}(\omega)$ holds.
\end{remark}

\vspace{0.5cm}

Indeed, the proof we just gave and Proposition \ref{prop3.1} lead to the following result.
\begin{theorem}
Let $1\leq p< \infty$, $\frac{1}{p} + \frac{1}{p'}= 1$ and $f\in L_{\rm loc}(\mathbb{R}^n)$. $|f|\in A_{p}$ if and only if $f\in \mathrm{BMO}^{p'}(|f|)$.
\end{theorem}

\vskip 0.5cm
\noindent
{\it Proof of Theorem \ref{main1.8}.}
Since the case $p=\infty$ is trivial, it suffices to consider $1\leq p<\infty$. Let $f\in \mathrm{BMO}^{p}(\omega)$. For any cube $Q$, we use Fubini's theorem and  Minkowski's inequality to establish
$$(U_{\psi}f)_{Q}=\int_{0}^{1}\bigg(\frac{1}{|Q|}\int_{Q}f(tx)dx\bigg)\psi(t)dt=\int_{0}^{1}f_{tQ}\psi(t)dt$$
and
\begin{eqnarray*}
&&\bigg(\frac{1}{w(Q)}\int_{Q}|U_{\psi}f(x)-(U_{\psi}f)_{Q}|^{p}\omega(x)^{1-p}dx\bigg)^{1/p}\\
&&\leq\bigg(\frac{1}{w(Q)}\int_{Q}\bigg(\int_{0}^{1}|f(tx)-f_{tQ}|\psi(t)dt\bigg)^{p}\omega(x)^{1-p}dx\bigg)^{1/p}\\
&&\leq\int_{0}^{1}\bigg(\frac{1}{\omega(Q)}\int_{Q}|f(tx)-f_{tQ}|^{p}\omega(x)^{1-p}dx\bigg)^{1/p}\psi(t)dt\\
&&=\int_{0}^{1}\frac{t^{(p-1)\alpha/p-n/p}}{\omega(Q)^{1/p}}\bigg(\int_{tQ}|f(x)-f_{tQ}|^{p}\omega(x)^{1-p}dx\bigg)^{1/p}\psi(t)dt\\
&&\leq\|f\|_{\mathrm{BMO}^{p}(\omega)}\int_{0}^{1}t^{(p-1)\alpha/p-n/p}\cdot\frac{\omega(tQ)^{1/p}}{\omega(Q)^{1/p}}\psi(t)dt\\
&&=\|f\|_{\mathrm{BMO}^{p}(\omega)}\int_{0}^{1}t^{(p-1)\alpha/p-n/p}\cdot
\bigg(\frac{\int_{tQ}\omega(x)dx}{\int_{Q}\omega(x)dx}\bigg)^{1/p}\psi(t)dt\\
&&=\|f\|_{\mathrm{BMO}^{p}(\omega)}\int_{0}^{1}t^{\alpha}\psi(t)dt.
\end{eqnarray*}
Therefore, we have obtained the upper estimate
$$\|U_{\psi}f\|_{\mathrm{BMO}^{p}(\omega)\rightarrow \mathrm{BMO}^{p}(\omega)}\leq \int_{0}^{1}t^{\alpha}\psi(t)dt.$$

To prove the opposite one, we take $f_{0}(x)=\omega(x)$. From Theorem \ref{main1.3} and the condition $\omega(tx)=t^{\alpha}\omega(x)$, it follows that $f_{0}(x)\in \mathrm{BMO}^{p}(\omega)$ and
$$U_{\psi}f_{0}(x)=\int_{0}^{1}f_{0}(tx)\psi(t)dt=f_{0}(x)\int_{0}^{1}t^{\alpha}\psi(t)dt.$$
If $U_{\psi}$ is a bounded operator from $\mathrm{BMO}^{p}(\omega)$ to itself, then $(1.2)$ holds. Meanwhile, if $(1.2)$ holds, the constant $\int_{0}^{1}t^{\alpha}\psi(t)dt$ is best possible. Thus we complete the proof. \qed

\vskip 0.5cm

Given a nonnegative function $\psi$ on $[0,1]$. For a measurable complex valued function $f$ on $\mathbb{R}^{n}$, the weighted Ces$\grave{a}$ro average operator $V_{\psi}$, the adjoint operator of $U_{\psi}$, is defined by
$$(V_{\psi}f)(x)=\int_{0}^{1}f(\frac{x}{t})t^{-n}\psi(t)dt.$$

The following result can be deduced immediately from the proof of Theorem \ref{main1.7}.

\begin{theorem}
Let $\psi: [0,1]\rightarrow [0,\infty)$ be a function, $ p\in [1, \infty]$, $\alpha\in \mathbb{R}$, $\omega\in A_{p'}$ and $\omega(tx)=t^{\alpha}\omega(x)$ for all $t\in (0,\infty)$.
Then,
 $V_{\psi}: \mathrm{BMO}^{p}(\omega)\rightarrow \mathrm{BMO}^{p}(\omega)$ exists as a bounded operator if and only if
$$\int_{0}^{1}t^{-\alpha-n}\psi(t)dt<\infty.\eqno(2.3)$$
Moreover, when (2.3) holds, the operator norm of $V_{\psi}$ on $\mathrm{BMO}^{p}(\omega)$ is given by
$$\|V_{\psi}\|_{\mathrm{BMO}^{p}(\omega)\rightarrow \mathrm{BMO}^{p}(\omega)}=\int_{0}^{1}t^{-\alpha-n}\psi(t)dt.$$
\end{theorem}

\section{Acknowledgement}

I would like to thank the referee for the helpful comment which improved the
presentation of this paper.

%%%%%%%%%%%%%%%%%%%%%%%%%%%%%%%%%%%%%%%%%%%%%%%%%%%%%%%%%%%%%%%%%%%%%%
%%%%%%%%%%%%%%%%%%%%%%%%%%%%%%%%%%%%%%%%%%%%%%%%%%%%%%%%%%%%%%%%%%%%%%
\color{black}

%\end{CJK}
\end{document}